\documentclass[reqno, twoside, 11pt, a4paper]{amsart}
\makeatletter

\@addtoreset{equation}{section}
\makeatother
\usepackage[dvipdfmx]{graphicx}
\usepackage[dvipdfm]{hyperref}
\usepackage{amsmath}
\usepackage{amssymb}
\usepackage{latexsym}
\usepackage{amsthm}
\usepackage{xcolor}
\newtheorem{Thm}{Theorem}[section]

\newtheorem{Lem}[Thm]{Lemma}
\newtheorem{Prop}[Thm]{Proposition}

\newtheorem{Cor}[Thm]{Corollary}

\theoremstyle{definition}
\newtheorem{Def}[Thm]{Definition}
\newtheorem{Ass}[Thm]{Assumption}

\newtheorem{Exa}[Thm]{Example}

\newcommand{\Z}{\mathbb{Z}}
\newcommand{\Zd}{\mathbb{Z}^d}
\newcommand{\Pp}{\mathbb{P}_p}
\newcommand{\PP}{\mathcal{P}} 
\newcommand{\U}{\mathcal{U}}  

\begin{document}

\title[Enlargement of subgraphs by percolation]{Enlargement of subgraphs of infinite graphs by Bernoulli percolation}
\author{Kazuki Okamura} 
\address{Research Institute for Mathematical Sciences, Kyoto University, Kyoto, 606-8502, JAPAN.} 
\email{kazukio@kurims.kyoto-u.ac.jp} 
\date{\today} 

\begin{abstract}
We consider changes in properties of a subgraph of an infinite graph 
resulting from the addition of open edges of Bernoulli percolation on the infinite graph to the subgraph.    
We give the triplet of an infinite graph, one of its subgraphs, and a property of the subgraphs.  
Then, in a manner similar to the way Hammersley's critical probability is defined, we can define two values associated with the triplet. 
We regard the two values as certain critical probabilities, and compare them  with Hammersley's critical probability. 
In this paper, we focus on the following cases of a graph property: 
being a transient subgraph, having finitely many cut points or no cut points, being a recurrent subset, or being connected. 
Our results depend heavily on the choice  of the triplet. 

Most results of this paper are announced in \cite{O16} without proofs.   
This paper gives full details of them.  
\end{abstract}

\maketitle

\setcounter{tocdepth}{2}
\tableofcontents

\section{Introduction and Main results}

A connected graph is called transient (resp. recurrent) if the simple random walk on it is transient (resp. recurrent).  
Benjamini, Gurel-Gurevich and Lyons \cite{BGGL} showed the cerebrating result claiming that the trace of the simple random walk on a transient graph is 
recurrent 
almost surely.  
If a connected subgraph of an infinite connected graph is transient, then the infinite connected graph is transient. 
Therefore, the trace is somewhat ``smaller" than the graph on which the simple random walk runs. 
Now we consider the following questions: 
How far are a transient graph $G$  and the trace of the simple random walk on $G$? 
More generally, how far are $G$ and a recurrent subgraph $H$ of $G$? 
How many edges of $G$ do we need to add to $H$ so that the enlargement  of $H$  becomes transient?  

There are numerous choices of edges of $G$ to be added to $H$.   
If we add finitely many edges  to $H$, then the enlarged graph is also recurrent.   
Therefore, we add {\it infinitely} many edges to $H$ and consider whether the enlarged graph is transient.    
In this paper, we add infinitely many edges of $G$ to $H$ {\it randomly}. 
Specifically, we add open edges of Bernoulli bond percolation on $G$ to $H$, 
and consider the probability that the enlargement of $H$  is transient.            

Now we state our framework. 
In this paper, a graph is a locally-finite simple graph. 
A simple graph is a non-oriented graph in which neither multiple edges or self-loops are allowed.   
$V(X)$ and $E(X)$ denote the sets of vertices and edges of a graph $X$, respectively.  
Denote the cardinality of $A \subset V(X)$ by $|A|$.  
If we consider the $d$-dimensional integer lattice $\Zd$,  
then it is the nearest-neighbor model. 

Let $G$ be an infinite connected graph.  
We say that a subgraph $H$ of $G$ is {\it connected} if for any two vertices $x$ and $y$ of $H$ there are vertices $x_0, \dots, x_n$ of $H$ 
such that $x_0 = x$, $x_n = y$, and  $\{x_{i-1}, x_{i}\} \in E(H)$ for each $i$. 
In this paper, we consider Bernoulli {\it bond} percolation and do not consider site percolation.   
Let $\Pp$ be the Bernoulli measure on the space of configurations of Bernoulli bond percolation on $G$ such that each edge of $G$  is open with probability $p \in (0,1)$.    
Denote a configuration of percolation by $\omega = (\omega_e)_{e \in E(G)} \in \{0,1\}^{E(G)}$. 
We say that an edge $e$ is open if $\omega_e = 1$ and closed otherwise. 
We say that an event $A \subset \{0,1\}^{E(G)}$ is increasing (resp. decreasing) if the following holds: 
if $\omega = (\omega_e) \in A$ and $\omega^{\prime}_{e} \ge  \omega_{e}$ (resp. $\omega^{\prime}_{e} \le  \omega_{e}$) for any $e \in E(G)$, then $\omega^{\prime} \in A$. 
Let $C_x$ be the open cluster containing $x \in V(G)$.  
We remark that $\{x\} \subset V(C_x)$ holds. 
By convention,  
we often denote the set of vertices $V(C_x)$ by $C_x$.  
Consider the probability that the number of vertices of $G$ connected by open edges from a fixed vertex is infinite under $\Pp$.  
Then Hammersley's critical probability $p_c (G)$ is the infimum of $p$ such that the probability is positive, that is, for some $x \in V(G)$, 
\[ p_{c}(G) = \inf\left\{p \in (0,1) : \Pp(|C_x| = +\infty) > 0 \right\}. \] 
This value does not depend on the choice of $x$. 

Similarly, we consider the probability that the enlarged graph is transient under $\Pp$ and either of the following two values: the infimum of $p$ such that the probability is positive, or the infimum of $p$ such that the probability is one.     
We regard these two values as certain critical probabilities, and compare them with Hammersley's critical probability. 
We also consider questions of this kind, not only for transience, but also for other graph properties.  

\begin{Def}[Enlargement of subgraph]\label{enl} 
Let $H$ be a subgraph of $G$.   
Let $\U(H) = \U_{\omega}(H)$ be a random subgraph of $G$ such that 
\[ V(\U(H)) := \bigcup_{x \in V(H)} V(C_x) \text{ and } E(\U(H)) := E(H) \cup \left(\bigcup_{x \in V(H)} E(C_x)\right). \]   
If $\omega$ is chosen according to $\Pp$, then, we write $\U(H) = \U_{p}(H) = \U_{p, \omega}(H)$. 
\end{Def}

If $H$ is connected, 
then $\U(H)$ is also connected. 
If $H$ consists of a single vertex $x$ with no edges, 
then $\U(H)$ is identical to $C_x$.      

In this paper, a {\it property} $\PP$ is an automorphism-invariant set of subgraphs of $G$. 
For ease of description,  %new added !!!!!!!  
we denote $X \in \PP$ (resp. $X \notin \PP$) 
if a subgraph $X$ of $G$ %delated; 
satisfies (resp. does not satisfy) $\PP$.  
Let $\mathcal{F}$ be the cylindrical $\sigma$-algebra on the configuration space $\{0,1\}^{E(G)}$. 

\begin{Ass} 
We assume that an infinite connected graph $G$, a subgraph $H$ of $G$, and a property $\PP$ satisfy the following: \\
(i) $G \in \PP$ and $H \notin \PP$. \\
(ii) The event that $\U(H) \in \PP$ is $\mathcal{F}$-measurable and increasing. 

If $H$ is chosen according to a probability law $(\Omega^{\prime}, \mathcal{F}^{\prime}, \mathbb{P}^{\prime})$, 
then  we assume that 
$H \notin \PP$ $\mathbb{P}^{\prime}$-a.s., and   
the event $\U(H) \in \PP$ is $\mathcal{F}^{\prime} \otimes \mathcal{F}$-measurable and increasing for $\mathbb{P}^{\prime}$-a.s.,   
where $\mathcal{F}^{\prime} \otimes \mathcal{F}$ denotes the product $\sigma$-algebra of $\mathcal{F}^{\prime}$ and $\mathcal{F}$.  
\end{Ass}

In Section 2, 
we will check that the event $\{\U(H) \in \PP\}$ is $\mathcal{F}$-measurable for those properties, 
and give an example of $(G, H, \PP)$ such that $\U(H) \in \PP$ is {\it not} $\mathcal{F}$-measurable.   

\begin{Def}[A certain kind of critical probability]   
\begin{equation*}
p_{c, 1}(G, H, \PP) := \inf\left\{p \in [0,1] : \Pp(\U_p (H) \in \PP) > 0 \right\}. 
\end{equation*}
\begin{equation*}
p_{c, 2}(G, H, \PP) := \inf\left\{p \in [0,1] : \Pp(\U_p (H) \in \PP) = 1  \right\}. 
\end{equation*} 
If $H$ obeys a law $\mathbb{P}^{\prime}$, then we define $p_{c, i}(G, H, \PP)$, $i = 1,2$, 
by replacing $\Pp$ above with the product measure $\mathbb{P}^{\prime} \otimes \Pp$ of $\mathbb{P}^{\prime}$ and $\Pp$.    
\end{Def}

The main purpose of this paper is to compare the values $p_{c, i}(G, H, \PP)$, $i = 1,2$, with $p_c(G)$.   
If $H$ is a single vertex and $\PP$ is being an infinite graph,  
then the definitions of $p_{c, 1}(G, H, \PP)$ and $p_c(G)$ are identical and, hence, $p_{c,1}(G, H, \PP) = p_c(G)$.   
It is easy to see that $p_{c,2}(G, H, \PP) = 1$. 

Before we proceed to main results, we introduce a series of notation and definitions. 
For a connected graph $X$,  $d_X(x, y)$ denotes the graph distance between $x$ and $y$ in $X$, and let 
\[ B_{X}(x, n) := \{y \in V(X) : d_{X}(x, y) \le n\}, \ x \in X, n \ge 0. \]  

We now briefly state the notion of Cayley graphs. 
Let $\Gamma$ be a finitely generated countable group and $\mathcal{S}$ be a symmetric finite generating subset of $\Gamma$ which does not contain the unit element.  
Then the {\it Cayley graph of $\Gamma$ with respect to $\mathcal{S}$} is the graph such that the set of vertices is $\Gamma$ and the set of edges is $\{\{x, y\} \subset \Gamma : x^{-1}y \in \mathcal{S}\}$.   
This graph depends on choices of $S$. 
In this paper, all results concerning Cayley graphs of groups do not depend on choices of $\mathcal{S}$.  
We say that a graph $G$ has the {\it degree of growth} $d \in (0, +\infty)$ if for any vertex $x$ of $G$, 
\[ 0 < \liminf_{n \to \infty} \frac{|B_G (x,n)|}{n^d} \le \limsup_{n \to \infty} \frac{|B_G (x,n)|}{n^d} < +\infty. \]

Let $\deg_G (x)$ be the number of edges containing a vertex $x$ of $G$.   
$((S_n)_{n \ge 0}, (P^x)_{x \in V(G)})$ denotes the simple random walk on $G$, specifically, the following hold for any $n \ge 0$ and any $x, y, z \in V(G)$:  
\[ P^x (S_{n+1} = z | S_n = y) = \frac{1}{\deg_G (y)}, \textup{ if } \{y, z\} \in E(G), \]   
\[ P^x (S_{n+1} = z | S_n = y) = 0, \textup{ otherwise.}  \] 
\[ P^x (S_0 = x) = 1. \]

\subsection{Main results}  

In this paper, we focus on each of the following properties: 
(i) being a transient subgraph,   
(ii) having finitely many cut points or having no cut points, 
(iii) being a recurrent subset, 
and  
(iv) being a connected subgraph.

\subsubsection{Being a transient graph}

Let $\PP$ be being a transient graph. 

\begin{Thm}\label{epsilon-thm}
Let $G = \Zd, d \ge 3$. 
Then for any $\epsilon > 0$ there is  a recurrent subgraph $H_{\epsilon}$ such that $p_{c, 2}(G, H_{\epsilon}, \PP) \le \epsilon$.  
\end{Thm} 

This will be shown in subsection 3.1.   Other results of the case that $H$ is a fixed subgraph will also be stated and proved.   

\begin{Thm}\label{trace3}
(i) Let $G$ be a Cayley graph of a finitely generated countable group with the degree of growth $d \ge 3$.  
Let $H$ be the trace of the simple random walk on $G$. 
Then \[ p_{c, 1}(G, H, \PP) \ge p_c(G).\]  
(ii) Let $G = \Zd, d \ge 3$.  If $H$ is the trace of the simple random walk on $\Zd$, 
then \[ p_{c,1}(\Zd, H, \PP) = p_{c,2}(\Zd, H, \PP) = p_c(\Zd). \]    
\end{Thm}

This will be shown in Subsection 3.2. 

\subsubsection{On the number of cut points}

We now consider the number of cut points. 
Let $P^{x, y}$ be the law of two independent simple random walks on $G$ which start at $x$ and $y$, respectively.  

\begin{Def}[cut point]\label{cut-def}
We say that a vertex $x \in V(G)$ is a {\it cut point}
if we remove an edge $e$ containing $x$,  
then  
the graph splits into two {\it infinite} connected components.   
\end{Def}

The graph appearing in the proof of Theorem \ref{extre-graph} (ii) (see Figure 3) has a vertex such that if we remove it, then the graph splits into two connected components. 
However, it is {\it not} a cut point in the sense of the above definition.  

\begin{Thm}\label{Thm-cut}
Let $G$ be a Cayley graph of a finitely generated countable group with  the degree of growth $d \ge 5$.
Let $H$ be the trace of the two-sided simple random walk on $G$.  
Then if $p < p_c(G)$, 
then  $\U_p (H)$ has infinitely many cut points, $P^{o,o} \otimes \Pp$-a.s. 
\end{Thm}

It is known that the trace of the two-sided simple random walk on $\Zd$ has infinitely many cut points $P^{0,0}$-a.s. (Cf. Lawler \cite[Theorem 3.5.1]{La})
The result above means that in the subcritical regime, there remain infinitely many cut points that  
are not bridged by open bonds of percolation. 
We give a proof of this assertion in Section 4. 

\subsubsection{Being a recurrent subset}

Now we consider the case that $\PP$ is being a recurrent subset.    
In this paper, we regard this as a subgraph and consider the induced subgraph of the subset. 
(See Diestel \cite{D} for the definition of  induced subgraphs.)
We now define recurrent and transient subsets of $G$ by following Lawler and Limic \cite[Section 6.5]{LL}.  
We regard a recurrent subset as a subgraph and consider the induced subgraph of the recurrent subset. 

\begin{Def}[recurrent subset]\label{Def-recur} 
We say that a subset $A$ of $V(G)$ is a {\it recurrent subset}  
if for some $x \in V(G)$ \[ P^x\left(S_n \in A \text{ i.o. } n\right) > 0. \]   
(Here and henceforth,  ``i.o." is an abbreviation of  ``infinitely many".)    
Otherwise, $A$ is called a {\it transient subset}.  
This definition does not depend on choices of a vertex $x \in V(G)$.    
\end{Def}

\begin{Thm}\label{recur-sub-main}
Let $G$ be a Cayley graph of a finitely generated countable group with  the degree of growth $d \ge 3$. 
Let $H$ be the trace of the simple random walk on $G$. 
Then  \\
(i) If $d \ge 5$,  
\[ p_{c, 1}(G, H, \PP) = p_{c, 2}(G, H, \PP) =  p_{c}(G).  \] 
(ii) If $d = 3, 4$, 
\[ p_{c, 1}(G, H, \PP) = p_{c, 2}(G, H, \PP) =  0.  \] 
\end{Thm}

This will be shown in Subsection 5.2. The case that $H$ is a fixed subgraph will be dealt with in Subsection 5.1. 

\subsubsection{Being connected}

By Definition \ref{enl}, if $H$ is connected, then $\U(H)$ is also connected. 
On the other hand, if $H$ is {\it not} connected, then $\U(H)$ can be non-connected. 
For example, if $(V(G), E(G)) = (\Z, \{n,n+1 : n \in \Z\})$ and $(V(H), E(H)) = (\Z, \emptyset)$, then  
\[ \Pp(\U(H) \textup{ is connected}) = 0, \ \ p < 1. \] 

The following is introduced by Benjamini, H\"aggstr\"om and Schramm  \cite{BHS}. 

\begin{Def}[percolating everywhere]\label{pe} 
We say that a subgraph $H$ of $G$ is {\it percolating everywhere}
if $V(H) = V(G)$ and every connected component of $H$ is infinite. 
\end{Def}

We introduce a notion concerning connectivity.  
For $A, B \subset V(G)$, 
we let \[ E(A, B) := \left\{\{y,z\} \in E(G) : y \in A, z \in B \right\}.\]   

\begin{Def}\label{sc}
We say that $G$ satisfies (TI)  
if for every $A, B \subset V(G)$ satisfying  
\[ V(G) = A \cup B, \ A \cap B = \emptyset \text{ and } |A| = |B| = +\infty,\]   
$E(A, B)$ is an infinite set.   
\end{Def}

The following are easy to see. %new; added
\begin{Exa}
(i) $\Zd, d \ge 2$, satisfy (TI).\\
(ii) $\mathbb{T}_d, d \ge 2$, does not satisfy (TI).\\
(iii) The trace of the two-sided simple random walk on $\Zd, d \ge 5$, does not satisfy (TI) a.s.  
\end{Exa}

The following concerns the connectedness of the enlargement of a percolating everywhere subgraph. 

\begin{Thm}\label{conn}
(i) If $G$ is (TI), then  for any percolating everywhere subgraph $H$ 
\[ p_{c,1}(G,H,\PP) = p_{c,2}(G,H,\PP). \]
If the number of connected components of $H$ is finite,  
then  \[ p_{c,1}(G,H,\PP) = p_{c,2}(G,H,\PP) = 0. \]
(ii) If $G$ does not satisfy (TI), then 
there is a percolating everywhere subgraph $H$ such that 
\[ p_{c,1}(G,H,\PP) = 0 \text{ and }  p_{c,2}(G,H,\PP) = 1. \]   
\end{Thm}    
%new; added (replaced)
We are not sure whether the following holds or not: 
if $G$ satisfies (TI) and $H$ is a percolating everywhere subgraph with infinitely many connected components    
then \[ p_{c, 1}(G, H, \PP) = p_{c, 2}(G, H, \PP) = 0. \]  

\subsection{Related results}

\cite{BHS} considers questions of this kind with a different motivation to ours. 
Their original motivation was 
considering 
the conjecture that for all $d \ge 2$, there is no infinite cluster in Bernoulli percolation on $\Zd$ with probability one at the critical point.   
If an infinite cluster of Bernoulli percolation $\mathcal{C}_{\infty}$ satisfies  
$p_{c}(\mathcal{C}_{\infty}) < 1$ $\Pp$-a.s. for any $p$,  
then  the conjecture holds.   
A question related to this is considering what kinds of conditions on a subgraph $H^{\prime}$ of $\Zd$ assure $p_c (H^{\prime}) < 1$.  
They introduced the concept of {\it percolating everywhere} (Recall Definition \ref{pe}.) and 
considered whether the following claim holds:  
if we add Bernoulli percolation to a percolating everywhere graph, 
then the enlarged graph is connected, and moreover, $p_{c}(\textup{the enlarged graph}) < 1$, $\Pp$-a.s. for any $p$.    
This case can be described using 
our terminology as follows.   
$G$ is $\Zd$, $H$ is a percolating everywhere subgraph, and  
$\PP$ is connected and moreover $p_c(\U(H)) < 1$.  
They showed that if $d = 2$, then $p_{c, 2}(G, H, \PP) = 0$, and   
conjectured that it also holds for all $d \ge 2$.   

Recently, Benjamini and Tassion \cite{BeTa} showed the conjecture for all $d \ge 2$  
by a method different from \cite{BHS}.   
Theorem \ref{conn} discusses the values $p_{c,i}(G, H, \PP), i = 1,2$, for percolating everywhere subgraphs $H$ of $G$. 
$G$ is not necessarily assumed to be $\Zd$, and the result depends on whether $G$ satisfies a certain condition. 

%new; added
Several researches deal with comparison of values of two kinds of thresholds of Bernoulli percolation.    
If $p > p_c (G)$, it is natural to ask how the number of infinite clusters vary as $p$ is increased, and  how many more
edges are needed in order for the infinite graphs to unite. 
Let \[ p_u (G) := \inf\{p \in [0,1] : \textup{ there exists a unique infinite cluster $\Pp$-a.s.}\}.\]  
Then, $p_c (G) \le p_u(G) \le 1$. 
It is natural to ask whether $p_c (G) < p_u (G)$ holds.  
By Burton-Keane \cite{BK}, $p_c (\Zd) = p_u (\Zd)$, $d \ge 2$. 
By their arguments, it follows that $p_c (G) = p_u (G)$ holds if $G$ is a vertex-transitive amenable graph.  
Benjamini-Schramm \cite{BS96} conjectured that $p_c (G) < p_u (G)$ holds if $G$ is a non-amenable vertex-transitive graph. 
In \cite{BS01} they gave a partial resolution of it, specifically, $p_c (G) < p_u (G)$ holds if $G$ is a non-amenable vertex-transitive planar graph with one end.   
This issue is more complicated for general graphs. 
As $p$ increases, there are the following two possibilities:   
On the one hand, since finite clusters can join up, new infinite clusters can be generated. 
On the other hand, since infinite clusters can join up, the number of infinite clusters can decrease. 
Benjamini \cite[Chapter 9]{B} and Lyons and Peres \cite[Chapter 7, Section 9]{LP} gave examples of graphs for which  it would be difficult to understand how the numbers of infinite clusters vary as $p$ increases.   
Comparison of $p_c (G)$ and $ p_u(G)$ is an intriguing problem, but in this paper we do not deal with this issue. 
See \cite[Chapter 9]{B} and \cite[Chapter 7, Section 9]{LP} for more details of this problem.

\subsection{Structure of paper}

The remainder of this paper is organized as follows.  
Section 2 states some preliminary results including the measurability of $\{\U(H) \in \PP\}$.      
We consider the case that $\PP$ is being a transient graph, 
the case that $\PP$ is a property concerning the number of cut points of graphs, 
the case that $\PP$ is being a recurrent subset, 
and the case that $\PP$ is being connected and $H$ is percolating everywhere, in Sections 3 to 6 respectively.

\section{Preliminaries}  

This section consists of three subsections. 
First we give a lemma estimating $p_{c, i}(G, H, \PP)$. 
Then we state some results concerning random walk and percolation. 
Finally we discuss the measurability of the event $\{\U(H) \in \PP\}$.

\subsection{A lemma} 

Roughly speaking, in the following, we will show that under a certain condition, $p_{c, i}(G, H, \PP)$ can be arbitrarily small, if there is a ``suitable" subgraph $H$.       
Let $\mathcal{N}(v)$ be the set of neighborhoods of a vertex $v$.    

\begin{Lem}\label{epsilon}
Fix an infinite connected graph $G$ and a property $\PP$ for subgraphs of $G$.  
Let $i = 1, 2$. 
Assume that there is  a subgraph $H$ of $G$ such that 
\[ p_{c, i}(G, H, \PP) < 1, \ \text{ and  } \]  
\begin{equation}\label{near} 
v \in V(H) \text{ or } \mathcal{N}(v)  \subset V(H), \ \text{ for any } v \in V(G).  
\end{equation}  
Then for any $\epsilon > 0$ 
there is  a subgraph $H_{\epsilon}$ such that  $p_{c, i}(G, H_{\epsilon}, \PP) \le \epsilon$.  
\end{Lem} 

\begin{proof} 
We show this assertion for $i=1$. 
Let $\Phi : \{0,1\}^{E(G)} \times \{0,1\}^{E(G)} \to \{0,1\}^{E(G)}$ be 
the map defined by \[ \Phi(\omega_{1}, \omega_{2}) = \omega_{1} \vee \omega_{2}.\]  
(Here and henceforth $\omega_{1} \vee \omega_{2}$ means the maximum of $\omega_1$ and $\omega_2$.)  
Then the push-forward measure of the product measure $\mathbb{P}_{q_1} \otimes \mathbb{P}_{q_2}$ on $\{0,1\}^{E(G)} \times \{0,1\}^{E(G)}$ 
by $\Phi$ is $\mathbb{P}_{q_1+q_2 - q_1q_2}$. 

Since $p_{c, 1}(G, H, \PP) < 1$, 
we have that for any $q_2 > 0$, 
there is  $q_1 < p_{c, 1}(G, H, \PP)$ 
such that \[ q_1+q_2 - q_1q_2 > p_{c, 1}(G, H, \PP).\]

It is easy to see that \[ \U_{\omega_{2}}(\U_{\omega_{1}}(H)) \subset \U_{\omega_{1} \vee \omega_{2}}(H).\]       
By \eqref{near}, 
\[ \U_{\omega_{2}}\left(\U_{\omega_{1}}(H)\right) = \U_{\omega_{1} \vee \omega_{2}}(H). \]

Therefore,  
\[ \mathbb{P}_{q_1} \otimes \mathbb{P}_{q_2}
\left( \U_{q_2, \omega_{2}}(\U_{q_1, \omega_{1}}(H)) \in \PP \right)  = \mathbb{P}_{q_1+q_2-q_1q_2}\left( \U_{q_1+q_2-q_1q_2} (H) \in \PP \right) > 0. \]

Since $q_1 < p_{c, 1}(G, H, \PP)$, 
there is  a configuration $\omega_{1}$ such that $\U_{\omega_{1}}(H) \notin \PP$ 
and \[ \mathbb{P}_{q_2}\left( \U_{q_2, \omega_{2}}(\U_{\omega_{1}}(H)) \in \PP \right) > 0.\]   
Hence \[ p_{c, 1}(G, \U_{\omega_{1}}(H), \PP) \le q_2.\]     
We can show this for $i=2$ in the same manner.   
\end{proof}

\subsection{Random walk and percolation}

Let $((S_n)_{n \ge 0}, (P^x)_{x \in V(G)})$ be the simple random walk on $G$.    
Let $T_A$ be the first hitting time of $(S_n)_n$ to a subset $A \subset V(G)$, that is, 
\[ T_A := \inf\{n \ge 0 : S_n \in A\}, \]    
where we let $\inf \emptyset = +\infty$.   

\begin{Lem}\label{Recur-cluster} 
Let $G$ be a Cayley graph of a finitely generated group with  the degree of growth $d$.    
Let $o$ be the unit element of the finitely generated group.   
Assume $p_c(G) < 1$ and $p \in (p_c(G), 1)$. 
Then,\\
(i) There is  a unique infinite cluster $\mathcal{C}_{\infty}$, $\Pp$-a.s. \\
(ii) $\mathcal{C}_{\infty}$ is a recurrent subset of $G$, that is, 
\begin{equation}\label{recur-posi}
P^{o}\left(S_n \in \mathcal{C}_{\infty} \text{ i.o. } n\right) > 0,  \ \text{ $\Pp$-a.s.} 
\end{equation} 
(iii)  
\begin{equation}\label{recur-one}
P^{o}\left(S_n \in \mathcal{C}_{\infty} \text{ i.o. } n\right) = 1, \ \text{ $\Pp$-a.s.}  
\end{equation}        
\end{Lem}

\begin{proof}
By Woess \cite[Theorem 12.2 and Proposition 12.4]{Wo}, Cayley graphs of a finitely generated group with polynomial growth is amenable graphs.  
Therefore, by \cite{BK}, %revised 
the number of infinite clusters is $0$ $\Pp$-a.s., or, it is $1$ $\Pp$-a.s.  
By $p > p_c(G)$, the latter holds. 
Thus we have assertion (i).  

We will show assertion (ii).  
Let $P^{o} \otimes \Pp$ be the product measure of $P^{o}$ and $\Pp$. 
\[ P^{o} \otimes \Pp\left(S_{n} \in \mathcal{C}_{\infty} \text{ i.o. } n\right)   
= \lim_{N \to \infty} P^{o} \otimes \Pp\left(\bigcup_{n \ge N}\{S_{n} \in \mathcal{C}_{\infty}\}\right). \]

Using the shift invariance of Bernoulli percolation and 
the Markov property for simple random walk,    
\[ P^{o} \otimes \Pp\left(\bigcup_{n \ge N}\{S_{n} \in \mathcal{C}_{\infty}\}\right)
= P^{o} \otimes \Pp\left(\bigcup_{n \ge 0} \left\{S_{N}^{-1} \cdot S_{N + n} \in \mathcal{C}_{\infty}\right\}\right)\]
\[ = P^{o} \otimes \Pp\left(\bigcup_{n \ge 0} \{S_{n} \in \mathcal{C}_{\infty}\}\right). \]
Here $S_{N}^{-1}$ is the inverse element of $S_N$ as group.   
Hence,
\[ P^{o} \otimes \Pp\left(S_{n} \in \mathcal{C}_{\infty} \text{ i.o. } n\right) 
= P^{o} \otimes \Pp\left(\bigcup_{n \ge 0} \{S_{n} \in \mathcal{C}_{\infty}\}\right). \]

Since 
\[ \{S_{n} \in \mathcal{C}_{\infty} \text{ i.o. } n\} \subset \bigcup_{n \ge 0} \{S_{n} \in \mathcal{C}_{\infty}\},\]   
we have 
\[ P^{o}\left(S_{n} \in \mathcal{C}_{\infty} \text{ i.o. } n\right) 
= P^{o}\left(\bigcup_{n \ge 0} \{S_{n} \in \mathcal{C}_{\infty}\}\right) 
= P^{o}\left(T_{\mathcal{C}_{\infty}} < +\infty\right) > 0, \Pp\text{-a.s.}  \]
Thus we have (\ref{recur-posi}).  

By \cite[Corollary 25.10]{Wo} all bounded harmonic functions on $G$ are constant.   
By following the proof of  \cite[Lemma 6.5.7]{LL}, we have (\ref{recur-one}).  
\end{proof}

\subsection{Measurability of $\U(H) \in \PP$}

Recall that $\mathcal{F}$ is the cylindrical $\sigma$-algebra of $\{0,1\}^{E(G)}$.  
First, we consider the case that $H$ is a non-random subgraph. 

\begin{Lem}
(i)  Let $H$ be a recurrent subgraph of a transient graph $G$. 
Then the event that $\U(H)$ is a transient subgraph of $G$ is $\mathcal{F}$-measurable. \\
(ii) Let $H$ be a recurrent subgraph of a transient graph $G$. 
Then the number of cut points of $\U(H)$ is an $\mathcal{F}$-measurable function.\\ 
(iii) Let $H$ be a transient subset of a transient graph $G$.  
Then the event that $\U(H)$ is a recurrent subset is $\mathcal{F}$-measurable. \\ 
(iv) Let $H$ be a non-connected subgraph of an infinite connected graph $G$. 
Then the event that $\U(H)$ is connected is $\mathcal{F}$-measurable. 
\end{Lem}

\begin{proof}
(i) Let $R_{\text{eff}}\left(x,  \U(H) \setminus B_{\U(H)}(x,n)\right)$ be the effective resistance from $x$ to the outside of $B_{\U(H)}(x,n)$. 
It suffices to show that \\
$R_{\text{eff}}\left(x, \U(H) \setminus B_{\U(H)}(x,n)\right)$ is an $\mathcal{F}$-measurable function for each $n$.   
Since $\U(H)$ is a connected subgraph of $G$, 
$B_{\U(H)}(x,n)$ is contained in $B_{G}(x,n)$.  
Therefore,    
$R_{\text{eff}}\left(x, \U(H) \setminus B_{\U(H)}(x,n)\right)$ is determined by configurations in $B_{G}(x,n)$ and hence is $\mathcal{F}$-measurable.        

(ii) It suffices to show that for any $z \in V(G)$, the event that $z \in \U(H)$ and $z$ is a cut point of $\U(H)$ is $\mathcal{F}$-measurable.    
$z$ is a cut point of $\U(H)$ if and only if  $z$ is a cut point of $\U(H \cap B_{G}(z, n))$ for any $n$.    

(iii)       
By Fubini's theorem, 
it suffices to see that $\{S_n \in \U(H)\}$ is $\mathcal{F}_{\text{SRW}} \otimes \mathcal{F}$-measurable. 
This follows from 
\[ \left\{S_n \in \U(H)\right\} = \bigcup_{y \in V(G)} \left\{S_n = y\right\} \times \left\{y \text{ is connected to } H \text{ by an open path}\right\}.\]       

(iv) If $x, y \in V(\U(H))$ are connected in $\U(H)$,   
then  there is  $n$ such that $x$ and $y$ are in a connected component of $\U(H) \cap B_{G}(x,n)$.  
This event is determined by configurations of edges in $B_{G}(x,n+1)$.         
\end{proof}

We now consider the case that $H$ is a random subgraph of $G$.   
Let $\mathcal{F}_{SRW}$ be the $\sigma$-algebra on the path space defined by the simple random walk on $G$ and $\mathcal{F}_{\text{SRW}} \otimes \mathcal{F}$ be the product $\sigma$-algebra of $\mathcal{F}_{\text{SRW}}$ and $\mathcal{F}$.  
The following easily follows from that the event that the trace of the simple random walk is identical with a given connected subgraph $H$ is $\mathcal{F}_{\text{SRW}}$-measurable.  

\begin{Lem}
Assume that the event $\U(H)$ satisfies $\PP$ is $\mathcal{F}$-measurable for any infinite connected subgraph $H$. 
Let $H$ be the trace of the simple random walk.  
Then the event $\U(H)$ satisfies $\PP$ is $\mathcal{F}_{\text{SRW}} \otimes \mathcal{F}$-measurable.   
\end{Lem}

\begin{Exa}[A triplet  $(G, H, \PP)$ such that the event $\{\U(H) \in \PP\}$ is not measurable]
We first show that there is a non-measurable subset of $\{0,1\}^{\mathbb{N}}$ with respect to the cylindrical $\sigma$-algebra of $\{0,1\}^{\mathbb{N}}$.    
Here and henceforth, $\mathbb{N}$ denotes the set of natural numbers. 
Let $\phi : \{0,1\}^{\mathbb{N}} \to \{0,1\}^{\mathbb{N}}$ be the one-sided shift and 
$A$ be an uncountable subset of $\{0,1\}^{\mathbb{N}}$ such that 
(i) \[ \bigcup_{n \ge 0} \phi^{-n}(A) = \{0,1\}^{\mathbb{N}} \setminus \left(\bigcup_{n \ge 1} \{x \in \{0,1\}^{\mathbb{N}} : \phi^n (x) = x\}\right),\] 
and 
(ii) for any $x, y \in A$ and any $n \ge 1$ $y \ne \phi^n (x)$.

Assume that $A$ is measurable. 
Let $\ell$ be the product measure of the probability measure $\mu$ on $\{0,1\}$ with $\mu(\{0\}) = \mu(\{1\}) = 1/2$. 
Since $\phi^{-i}(A) \cap \phi^{-j}(A) = \emptyset$ for $i \ne j$,  
\[ \ell\left(\bigcup_{i \ge 0} A\right) = \sum_{i \ge 0} \ell(\phi^{-i}(A)). \]   
Since $\cup_{n \ge 1} \{x \in \{0,1\}^{\mathbb{N}} : \phi^n (x) = x\}$ is countable, 
$\ell(\cup_{i \ge 0} A) = 1$.  
Since $\phi$ preserves $\ell$, 
we see that 
\[ \ell(\phi^{-i}(A)) = \ell(A)\] 
for any $i$, and 
\[ \sum_{i \ge 0} \ell(\phi^{-i}(A)) = 0 \text{ or } +\infty.\]     
But this is a contradiction. 
Hence $A$ is not measurable.

Let $G$ be the connected subgraph of $\mathbb{Z}^2$ whose vertices are \[ \{(x, 0) : x \ge -2\} \cup \{(y, 1) : y \ge -1\} \cup \{(-1, -1)\}.\]   
Then any graph automorphism of $G$ is the identity map between vertices of $G$.  
Let $H$ be the connected subgraph of $G$ whose vertices are \[ \{(x, 0) : x \ge -2\} \cup \{(-1, 1)\} \cup \{(-1, -1)\}.\]  
Then \[ E(G) \setminus E(H) = \{\{(n, 0), (n, 1)\} : n \in \mathbb{N}\}.\]  
Let $\widetilde\omega$ be the projection of $\omega \in \{0,1\}^{E(G)}$ to $\{0,1\}^{E(G) \setminus E(H)}$. 
Regard $E(G) \setminus E(H)$ as $\mathbb{N}$.      
Let $\PP$ be the property that a graph is isomorphic to a graph in the class $\{\U_{\omega}(H) : \widetilde\omega \in A\}$.  
Then \[ \{\U(H) \in \PP\} = A \times \{0,1\}^{E(H)}.\]  
This event is not measurable with respect to the cylindrical $\sigma$-algebra of $\{0,1\}^{E(G)}$.        
\end{Exa}

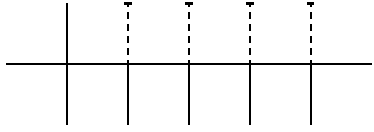
\begin{figure}[htbp]
\centering
\setlength\unitlength{0.8mm}
\begin{picture}(70, 30)
\put(0, 15){\line(10, 0){10}}
\put(10, 15){\line(0, 10){10}\line(10, 0){10}}
\put(10, 5){\line(0, 10){10}}
\put(20, 15){\dashbox(0, 10)}
\put(20, 15){\line(10, 0){10}} 
\put(20, 5){\line(0, 10){10}}
\put(30, 15){\dashbox(0, 10)}
\put(30, 15){\line(10, 0){10}} 
\put(30, 5){\line(0, 10){10}}
\put(40, 15){\dashbox(0, 10)}
\put(40, 15){\line(10, 0){10}} 
\put(40, 5){\line(0, 10){10}}
\put(50, 15){\dashbox(0, 10)}
\put(50, 15){\line(10, 0){10}} 
\put(50, 5){\line(0, 10){10}}
\end{picture}
\caption{\small Graph of $H$. The dotted lines are $E(G) \setminus E(H)$.} 
\end{figure}

%%%%%%%%%%%%%%%%%%%%%%%%%%%%%%%%%%%%%%%%%%%%%%%%%%%%%%%%%%%%%%%

\section{$\PP$ is being a transient graph}

In this section, we consider the case that $\PP$ is being a transient graph and assume that $H$ is connected.  

\subsection{The case that $H$ is a fixed subgraph}

\begin{Thm}[Extreme cases]\label{extre-graph}
(i) There is  a graph $G$ such that for any recurrent subgraph $H$ %new !!!!! 
\[ 0 < p_{c}(G) < p_{c, 1}(G, H, \PP) = 1.\]  
(ii) There is  a graph $G$ such that for any infinite recurrent subgraph $H$ of $G$, 
\[ p_{c,2}(G, H, \PP) = 0.\]   
\end{Thm}

We remark that if $H$ is finite, then $p_{c,2}(G, H, \PP) = 1$. 

\begin{proof} 
(i) 
Let $G$ be the graph which is constructed as follows : 
Take $\mathbb{Z}^{2}$ and attach a transient tree $T$ 
such that $p_{c}(T) = 1$ to the origin of $\mathbb{Z}^{2}$.
This appears in H\"aggstr\"om and Mossel \cite[Section 6]{HM}.   
Then for any recurrent subgraph $H$,  
\[ p_{c}(G) < p_{c, 1}(G, H, \PP) = 1.\]

Let $p < 1 = p_c(T)$.   
$\U_p(H)$ is the graph obtained by the union of $\U_p(H \cap \mathbb{Z}^{2})$
and $\U_p(H \cap T)$. 
$\U_p(H \cap \mathbb{Z}^{2})$ is recurrent.  
Then,  $\Pp$-a.s., 
$\U_p(H \cap T)$ is the graph obtained by adding at most countably many finite graphs to $H \cap T$. 
Hence $\U_p(H \cap T)$ is also recurrent, $\Pp$-a.s.  
Since the intersection of $\U_p(H \cap T)$  and $\U_p(H \cap \mathbb{Z}^{2})$ is the origin,  
$\U_p(H)$ is recurrent, $\Pp$-a.s.  

(ii) 
Let $G$ be an infinite connected line-graph in Benjamini and Gurel-Gurevich \cite[Section 2]{BGG}. 
In their paper, it is given as a graph having multi-lines, 
but we can construct a simple graph by adding a new vertex on each edge.

Let $H$ be an infinite connected recurrent subgraph of $G$.  
Then $\mathbb{N} \subset V(H)$.  
Let $p > 0$.  
If the number of edges between $k$ and $k+1$ is $2k^3$, then 
\[ \lim_{k \to \infty} \Pp\left( \left| \text{two open consecutive edges connecting $k$ and $k+1$} \right| > k^2 \right) = 1 \]
and the convergence is exponentially fast.  
Hence 
\[ \Pp\left(\bigcap_{k \ge 1} \left\{ \left|\text{two open consecutive edges connecting $k$ and $k+1$}\right| > k^2 \right\} \right) > 0. \]
By this and the recurrence/transience criterion by effective resistance (See \cite[Theorem 2.12]{Wo} for example.),     
\[ \Pp(\U_p (H) \text{ is transient}) > 0.\]          
Since $H$ is infinite, we can use the 0-1 law and have \[ \Pp(\U_p (H) \text{ is transient}) = 1.\]    
\end{proof}

We give rough figures of the two graphs in the proof above.     

\begin{figure}[htbp]
\begin{minipage}{0.45\hsize}
\centering
\includegraphics[width = 6cm, height = 4cm, bb = 0 0 800 600]{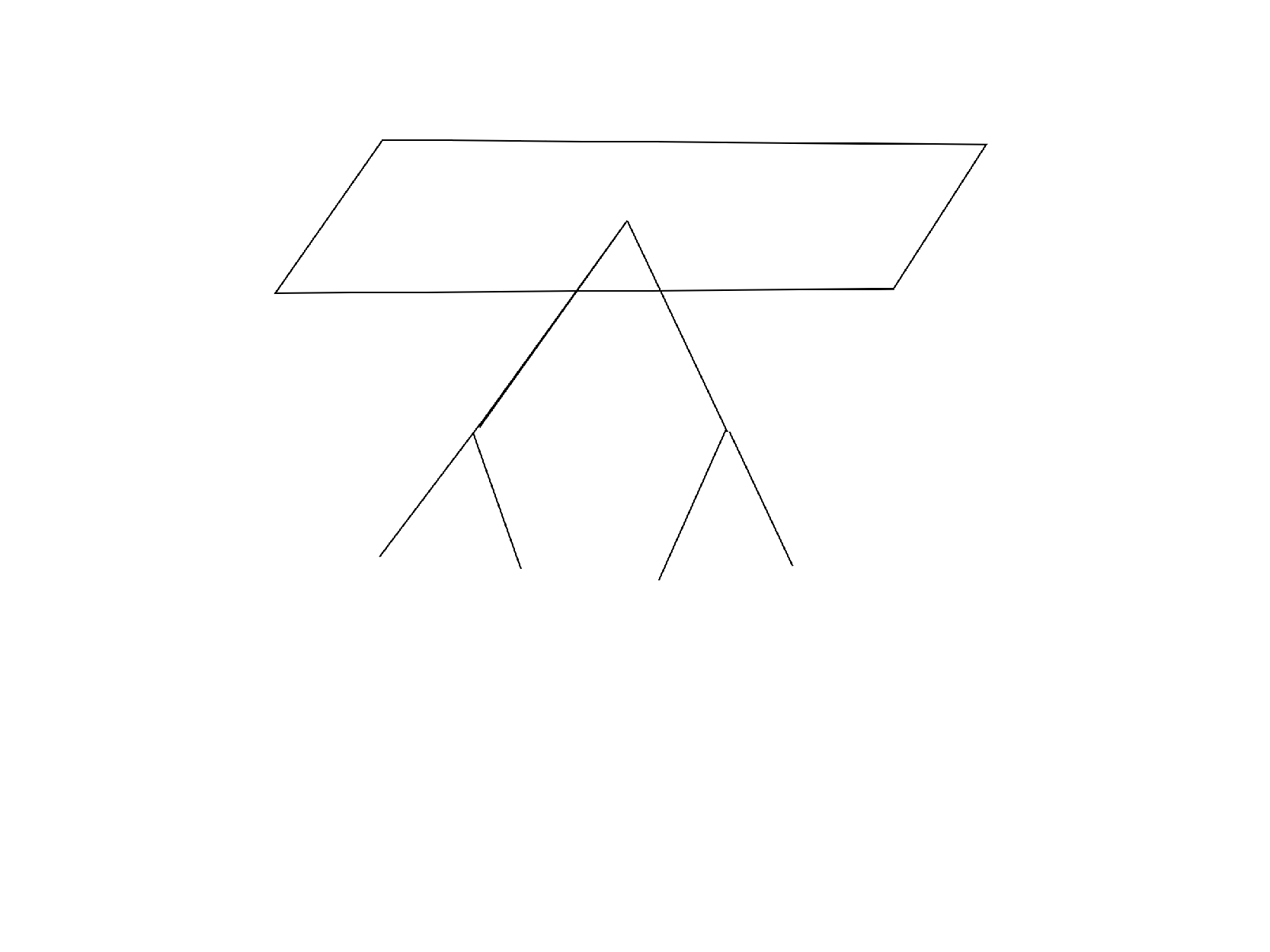}
\caption{\small Graph in the proof of Theorem \ref{extre-graph} (i)} 
\end{minipage}
\begin{minipage}{0.45\hsize}
\centering
\includegraphics[width = 6cm, height = 4cm, bb = 0 0 800 600]{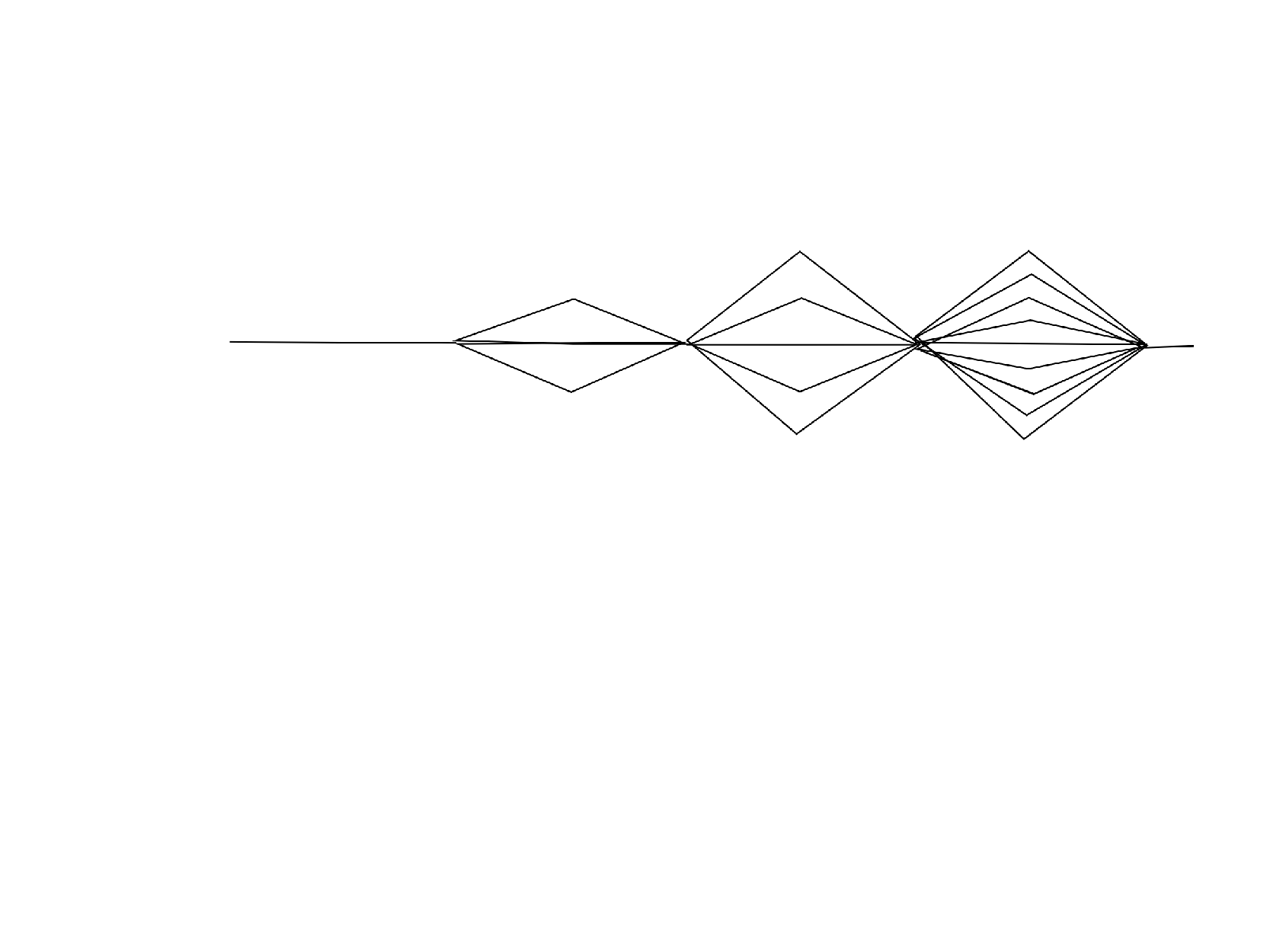}  
\caption{\small Graph in the proof of Theorem \ref{extre-graph} (ii)}
\end{minipage}
\end{figure}

The proof of (ii) above heavily depends on the fact that $G$ has unbounded degrees. 
Now we consider a case that $G$ has bounded degrees.

\begin{proof}[Proof of Theorem \ref{epsilon-thm}] 
Let $H$ be a recurrent subgraph of $\Zd$ such that $V(H) = V(G)$. 
By \cite[(2.21)]{Wo} such $H$ exists. 
If $p > p_{c}(\Zd)$, 
then $\U_p(H)$ contain the unique infinite open cluster $\Pp$-a.s.  
By Grimmett, Kesten and Zhang \cite{GKZ}, $\U_p(H)$ is transient $\Pp$-a.s.  
Hence \[ p_{c, 2}(G, H, \PP) \le p_{c}(\Zd) < 1.\]  
Now the assertion follows from this and Lemma \ref{epsilon}.   
\end{proof}

In the proof of Theorem \ref{epsilon-thm}, we choose a subgraph $H$ such that $V(H) = V(G)$ and apply Lemma \ref{epsilon}.   
However, if $H$ is a connected proper subgraph of an infinite tree $T$ with $\deg(x) \ge 2, \forall x \in V(T)$,    
then \eqref{near} in Lemma \ref{epsilon} fails.

\begin{Thm}\label{tree-graph}  
Let $T$ be an infinite transient tree. Then \\
(i) If $T^{\prime}$ is a recurrent subtree of $T$, then 
\[ p_{c, 1}(T, T^{\prime}, \PP) = p_{c}(T).\] 
(ii) If $T^{\prime}$ is an infinite recurrent subtree of $T$, then 
\[ p_{c,2}(T, T^{\prime}, \PP) = \sup\left\{p_{c}(H) : H \text{ is a transient subtree in } T \text{ and } E(H) \cap E(T^{\prime}) = \emptyset  \right\}. \]
\end{Thm}  

\begin{proof}
(i) 
By Peres \cite[Exercise 14.7]{P}, 
if $p > p_{c}(T)$, then, 
\[ \Pp(C_{v} \text{ is transient}) > 0, \] for any $v \in T$. 
Since $\Pp$-a.s. $C_{v} \subset \U_p(T^{\prime})$ for any $v \in T$, 
we have that 
\[ \Pp(\U_p(T^{\prime}) \text{ is transient}) > 0.\]  
Therefore, $p_{c, 1}(T, T^{\prime}, \PP) \le p_{c}(T)$. 

If $p < p_{c}(T)$, 
then, $\Pp$-a.s., 
$\U_p(T^{\prime})$ is an infinite tree obtained by attaching at most countably many finite trees to $T^{\prime}$.
Hence, $\U_p(T^{\prime})$ is also a recurrent graph $\Pp$-a.s. 
Therefore, $p_{c, 1}(T, T^{\prime}, \PP) \ge p_{c}(T)$.

(ii) 
Assume that there is  a transient subtree $H$ of $T$ such that $E(H) \cap E(T^{\prime}) = \emptyset$ and  
$p < p_c(H)$. 
There is  a finite path from $o$ to a vertex of $H$. 
Since $H$ is transient, 
the probability that random walk starts at $o$ and, then, goes to a vertex of $H$ and remains in $H$ after the hitting to $H$ is positive.  
Hence 
the probability that $\U_p(T^{\prime})$ is recurrent under $\Pp$ is positive. 
Hence $p \le p_{c,2}(T, T^{\prime}, \PP)$.   

Assume that 
\[ p > \sup\left\{p_{c}(H) : H \text{ is a transient subtree in } T \text{ and } E(H) \cap E(T^{\prime}) = \emptyset    \right\}.\] 
Since there are infinitely many transient connected subtrees $H$ of $T$ such that $E(H) \cap E(T^{\prime}) = \emptyset$,     
$\U_p(T^{\prime})$ contains at least one infinite transient cluster in $H$, $\Pp$-a.s.   
\end{proof}

Hereafter $\mathbb{T}_d$ denotes the $d$-regular tree, $d \ge 2$.   
By Theorem \ref{tree-graph}, 
\begin{Cor}
Let $T = \mathbb{T}_d, d \ge 3$ and $T^{\prime}$ be a recurrent subgraph. 
Then 
\[ p_{c, 1}(T, T^{\prime}, \PP) = p_{c, 2}(T, T^{\prime}, \PP) = p_c(T).\]      
\end{Cor}

The value $p_{c, 2}$ depends on choices of a subgraph $T^{\prime}$ as the following example shows.  
\begin{Exa}
Let $T$ be the graph obtained by attaching a vertex of $\mathbb{T}_3$ to a vertex of $\mathbb{T}_4$. \\
(i) If $T^{\prime}$ is a subgraph of $\mathbb{T}_3$ which is isomorphic to $L = \left(\mathbb{N}, \{\{n, n+1\} : n \in \mathbb{N}\}\right)$,  
then \[ p_{c, 2}(T, T^{\prime}, \PP) = p_c(\mathbb{T}_3) = \frac{1}{2}.\]  
(ii) If $T^{\prime}$ is a subgraph of $\mathbb{T}_4$ which is isomorphic to $L$,  
then \[ p_{c, 2}(T, T^{\prime}, \PP) = p_c(\mathbb{T}_4) = \frac{1}{3}.\]  
\end{Exa}

We give a short remark about stability with respect to rough isometry.  
Let $G$ be the graph obtained by attaching one vertex of the triangular lattice to a vertex of the $d$-regular tree $\mathbb{T}_d$.     
If $d = 3$, then \[ p_c(G) = p_c(\text{triangular lattice}) = 2 \sin \left(\frac{\pi}{18}\right) < \frac{1}{2} \] and 
\[ p_{c,1}(G, H, \PP) = p_c(\mathbb{T}_3) = \frac{1}{2}. \]      
If $d$ is large, then    
\[ p_c(G) = p_{c, 1}(G, H, \PP) = \frac{1}{d-1}. \]     

As this remark shows, there is a pair $(G, H)$ such that \[ p_c(G) < p_{c,1}(G, H, \PP) < 1.\]   
We are not sure that there is a pair $(G, H)$ such that  \[ 0 < p_{c,1}(G, H, \PP) < p_c(G).\]

\subsection{The case that $H$ is the trace of the simple random walk} 
   
Hereafter $E^{\mu}$ denotes the expectation with respect to a probability measure $\mu$.  

\begin{proof}[Proof of Theorem \ref{trace3}] 
Let $p < p_c(G)$. 
We show that the volume growth of $\U_p(H)$ is (at most) second order.  
We assume that simple random walks start at a vertex $o$. 
Since $B_{\U_p(H)}(o, n)$ is contained in $B_{G}(o, n)$, 
\[ B_{\U_p(H)}(o, n) \subset \bigcup_{x \in V(H) \cap B_{G}(o, n)} C_x. \]
By Mensikov \cite{M},  %corrected 
\[ E^{\Pp}\left[ \left| C_o \right| \right] < +\infty, \ \ p < p_c(G).\]   
Therefore,     
\[ E^{P^o \otimes \Pp}\left[ |B_{\U_p(H)}(o, n)| \right] \le E^{\Pp}[|C_o|] E^{P^o}\left[\left|V(H) \cap B_{G}(0, n)\right|\right]. \]
Using the Gaussian upper bounds for heat kernel obtained by Hebisch and Saloff-Coste \cite[Theorem 5.1]{HSC} and summation by parts, 
\begin{align*} 
E^{P^o}\left[ \left| V(H) \cap B_{G}(o, n) \right| \right] &\le \sum_{x \in B_{G}(o, n)} \sum_{m \ge 0} P^{o}(S_m = x) \\
&= \sum_{x \in B_{G}(o, n)} d_G(o,x)^{2-d} = O(n^2). 
\end{align*}  
Using this and Fatou's lemma,
\[ E^{P^o \otimes \Pp}\left[\liminf_{n \to \infty} \frac{| B_{\U_p (H)}(o, n) |}{n^2} \right] < +\infty. \]
Hence 
\[ \liminf_{n \to \infty} \frac{\left| B_{\U_p (H)}(o, n) \right|}{n^2} < +\infty, \ \ P^o \otimes \Pp\text{-a.s.} \]      
Assertion (i) follows from this and \cite[Lemma 3.12]{Wo}.  

Let $G = \Zd$, $d \ge 3$ and $H$ be the trace of the simple random walk on $G$.    
Then by Lemma \ref{Recur-cluster} and the transience of infinite cluster by \cite{GKZ}, 
\[ p_{c, 2}(G, H, \PP) \le p_c(G).\]   
Using this and assertion (i),  assertion (ii) follows. 
\end{proof}

%%%%%%%%%%%%%%%%%%%%%%%%%%%%%%%%%%%%%%%%%%%%%%%%%%

\section{$\PP$ is a property concerning cut points}

In this section, we assume that $G$ is a transient graph and $H$ is a recurrent subgraph of $G$. 

\begin{proof}[Proof of Theorem \ref{Thm-cut}]
Fix a vertex $o$.    
First we will show that 
\begin{equation}\label{cut-1}  
P^{o,o} \otimes \Pp\left(\text{$o$ is a cut point of $\U_p (H)$}\right) > 0.
\end{equation}  

We give a rough sketch of proof of (\ref{cut-1}). 
First we show there exists a vertex $z$ such that two simple random walks starting at $o$ and $z$ respectively do not intersect with positive probability.  
Then we ``make" vertices in a large box closed and show the two random walks do not return to the large box with positive probability. 
Finally we choose a path connecting the two traces in a suitable way.  

Let \[ S^i (A) := \{S^i_n : n \in A\} \text{ for } A \subset \mathbb{N}, \ i = 1,2.\]  
Using the finiteness of mean size of open cluster in the subcritical phase by \cite{M} 
and the Gaussian upper bounds for heat kernel obtained by \cite[Theorem 5.1]{HSC},     %explanations added
\begin{align*} 
\sum_{i, j \ge 0} P^{o,o} \otimes \Pp\left(\U_p(\{S^1_{i}\}) \cap \U_p(\{S^2_{j}\}) \ne \emptyset\right)  
&= \sum_{x \in V(G)} \sum_{i, j \ge 0} P^{o}(S_{i+j} = x) \Pp(x \in C_o) \\
&\le \left(\sup_{x \in V(G)} \sum_{k \ge 0} k P^{o}(S_{k} = x)\right) E^{\Pp} \left[\left| C_o \right|\right] < +\infty. 
\end{align*} 

Hence for large $N$,   
\[ P^{o,o} \otimes \Pp\left(\U_p \left(S^{1}\left([N, +\infty)\right)\right) \cap \U_p \left(S^2\left([N,+\infty)\right)\right) \ne \emptyset \right) \]
\[ \le \sum_{i, j \ge N} P^{o,o} \otimes \Pp\left(\U_p(\{S^1_{i}\}) \cap \U_p(\{S^2_{j}\}) \ne \emptyset\right) < 1. \]

Let \[ A := \left\{\U\left(S^{1} \left( [1, +\infty) \right)\right) \cap \U\left(S^{2}([0,+\infty))\right) = \emptyset \right\}.\]  
Then there is  a vertex $z \in V(G)$ such that $P^{z, o}(A) > 0$. 
If $z = o$, then (\ref{cut-1}) holds. 
Assume $z \ne o$. 
Let $B := B_{G}(o, 3d_G(o, z))$ and $C$ be the event that all edges in $B$ are closed. 

Since $p < 1$ and $A$ is decreasing,   
\[ P^{z,o} \otimes \Pp \left(A \cap C\right) > 0. \]

Since $S^1$ and $S^2$ are transient, 
there is  $N$ such that  
\[ P^{z,o} \otimes \Pp \left(A \cap C \cap \bigcap_{i = 1,2}\left\{S^i((N, \infty)) \cap B = \emptyset\right\} \right) > 0. \]

Now we can specify two finite paths of $S^i_j$.  
There are vertices $x^{i}_j, i=1,2, j = 0, \cdots, N$ such that 
\[ P^{z,o} \otimes \Pp \left(A \cap C \cap \bigcap_{i = 1,2}\left\{S^i_j = x^{i}_j, \forall j,  S^i((N, \infty)) \cap B = \emptyset\right\} \right) > 0. \]

Now we can pick up a path in $B$ connecting $\{x^{1}_j\}_{j} \cap B$ and $\{x^{2}_j\}_{j} \cap B$.   
We can let $S^i_j = x^i_j$, for $-m_i < j < 0$, $i = 1,2$, 
and 
$x^1_{-m_1} = x^2_{-m_2} = y_0$. 
\[ P^{y_0, y_0} \otimes \Pp \left(A \cap C \cap \bigcap_{i = 1,2}\left\{S^i_j = x^{i}_j,  \forall j,  S^i((N, \infty)) \cap B = \emptyset\right\} \right) > 0. \]
This event is contained in the event $\left\{\U (S^1([-m_1, +\infty)))  \cap \U (S^2([-m_2, +\infty))) = \emptyset\right\}$ and hence we have (\ref{cut-1}).   

Let $\mathcal{S}$ be the generating set of the Cayley graph $G$.   
Let $a e := \{a x, a y\}$ for an edge $e = \{x,y\}$ and a point $a \in \mathcal{S}$.     
Consider the following transformation $\Theta$ on $\mathcal{S}^{\mathbb{Z}} \times \{0,1\}^{E(G)}$ defined by 
\[ \Theta\left((a_j)_j, (\omega_e)_e\right) := \left((b_{j})_j, (\omega_{a_0 e})_e \right), \]
where we let $b_j := j+1$ for any $j \in \mathbb{Z}$.      
Then it follows that $\Theta$ preserves $P^{o,o} \otimes \Pp$.  

By applying the Poincar\'e recurrence theorem (See Pollicott and Yuri \cite[Theorem 9.2]{PY} for example)     
to $(\mathcal{S}^{\mathbb{Z}} \times \{0,1\}^{E(G)}, P^{o,o} \otimes \Pp, \Theta)$,  
we have 
\[ P^{o,o} \otimes \Pp \left(\U_p(\widetilde S((-\infty, n]))  \cap \U_p(\widetilde S([n+1, +\infty))) = \emptyset \text{ infinitely many } n \in \mathbb{Z} \right) > 0, \]%revised 
where we let \[ \widetilde S_n := \begin{cases} S^{1}_{n}, \ \  n \ge 0 \\ S^2_{-n}  \ \ n < 0.\end{cases}\]     

Define a transformation $\varphi_a$ on $\{0,1\}^{E(G)}$ by 
\[ \varphi_a \left((\omega_e)_e\right) := (\omega_{ae})_e.\]   
By following the proof of Bollob\'as and Riordan \cite[Lemma 1 in Chapter 5]{BR},  
the family of maps $\{\varphi_a : a \in \mathcal{S}\}$ is ergodic. 
By Kakutani \cite[Theorem 3]{Ka},    
$\Theta$ is ergodic with respect to $P^{o,o} \otimes \Pp$. 

Since the event that $\U_p(\widetilde S((-\infty, n]))  \cap \U_p(\widetilde S([n+1, +\infty))) = \emptyset$ holds for infinitely many $n \in \mathbb{Z}$ is $\Theta$-invariant, 
we have 
\[ P^{o,o} \otimes \Pp \left(\U_p(\widetilde S((-\infty, n]))  \cap \U_P(\widetilde S([n+1, +\infty))) = \emptyset \text{ i.o. } n \in \mathbb{Z} \right) = 1. \] 
\end{proof}

The following considers this problem at the critical point in high dimensions.  
It is pointed out by Itai Benjamini. (personal communication) 

\begin{Thm}\label{Itai}
Let $G = \Zd, d \ge 11$.  
Let $H$ be the trace of the two-sided simple random walk on $\Zd$.  
Let $p = p_c(\Zd)$. 
Then $\U_{p}(H)$ has infinitely many cut points $P^{o,o} \otimes \mathbb{P}_{p_c(\Zd)}$-a.s. 
\end{Thm}

\begin{proof}
We will show that 
\[ \sum_{x \in \Zd} \sum_{i, j \ge 0} P^{0}(S_{i+j} = x) \Pp(x \in C_0) < +\infty. \]
In below $c$, $c^{\prime}$ and $c^{\prime\prime}$ are constants depending only on $d$ and $p$.   %new 
Fitzner and van der Hofstad \cite[Theorem 1.4]{FvdH} 
claims that the decay rate for the two-point function $\mathbb{P}_{p_c(\Zd)}(0 \leftrightarrow x)$ is $|x|^{2-d}$ as $|x| \to +\infty$ for $d \ge 11$,  
by verifying several conditions in Hara's paper \cite{H}, which shows this decay rate for $d \ge 19$.     
Therefore, 
\[ \Pp(x \in C_0) \le c \cdot d_{\Zd}(0, x)^{2-d}, \ \ \text{ for any $x$.} \]  
Since $P^{0}(S_k = x) = 0$ if $k < d_{\Zd}(o,x)$, 
\begin{align*} 
\sum_{x \in \Zd} \sum_{i, j \ge 0} P^{0}(S_{i+j} = x) \Pp(x \in C_0) &\le \sum_{k} c^{\prime} k^{1 - d/2} \left(\sum_{l = 1}^{k} c l^{2-d} \left|\left\{x : d_{\Zd}(0,x) = l\right\}\right| \right)\\ 
&= c^{\prime\prime} \sum_{k \ge 1} k^{3-d/2} < +\infty. 
\end{align*}  
The rest of the proof goes in the same way as in the proof of Theorem \ref{Thm-cut}.   
\end{proof}

The following deals with the supercritical phase. 

\begin{Prop} 
Let $G = \Zd, d \ge 3$.  
Let $H$ be the trace of the two-sided simple random walk on $\Zd$. 
If $p > p_c(G)$, then $\U_p(H)$ has no cut points $P^{0,0} \otimes \Pp$-a.s. 
\end{Prop}

\begin{proof}
Using the two-arms estimate by Aizenman, Kesten and Newman \cite{AKN}, 
\[ \Pp \left(\text{$0 \in \mathcal{C}_{\infty}$ and $0$ is a cut point of $\mathcal{C}_{\infty}$}\right) = 0. \]
Using the shift invariance of $\Pp$,   
the unique infinite cluster $\mathcal{C}_{\infty}$ has no cut points $\Pp$-a.s.  
\end{proof}

%%%%%%%%%%%%%%%%%%%%%%%%%%%%%%%%%%%%%%%%%%%%%%%%%%%%%

\section{$\PP$ is being a recurrent subset}

In this section, we assume that $G$ is a transient graph. 
Recall Definition \ref{Def-recur}.    
We regard a recurrent subset as a subgraph and consider the induced subgraph of the recurrent subset.
In other words, if $A$ is a recurrent subset of $V(G)$, then we consider the graph such that the set of vertices is $A$ and the set of edges $\{\{x, y\} \in E(G) : x, y \in A\}$.  

\subsection{The case that $H$ is a fixed subgraph}

We proceed with this subsection as in Subsection 3.1. 
The following correspond to Theorem \ref{extre-graph}.  

\begin{Thm}[Extreme cases]\label{extre-subset}
(i) There is  a graph $G$ such that  for any transient subset $H$ of $G$, \[ 0 < p_{c}(G) < p_{c, 1}(G, H, \PP) = 1.\]  
(ii) There is  a graph $G$ such that for any infinite transient subset $H$ of $G$, \[ p_{c,2}(G, H, \PP) = 0.\]   
\end{Thm}

We show this in the same manner as in the proof of Theorem \ref{extre-graph}.  

\begin{proof}  
Even if we add one edge to a transient subset, 
then  the enlarged graph is also a transient subset. 
If not, the random walks hit an added vertex infinitely often, a.s., which contradicts that $G$ is a transient graph. 
Therefore, we can show (i) in the same manner as in the proof of Theorem \ref{extre-graph} (i).

Let $G$ be the graph defined in the proof of Theorem \ref{extre-graph} (ii).  
Let $p > 0$. 
Then, 
\[ \left|\mathbb{N} \cap V(\U_p(H))\right| = +\infty, \ \ \Pp \text{-a.s.}  \] 
Hence $\U_p (H)$ is a recurrent subset, $\Pp$-a.s.   
\end{proof}

Second we consider the case $G$ is $\Zd, d \ge 3$.   
Lemma \ref{Recur-cluster} implies that 
\begin{Prop}
Let $G = \Zd$, $d \ge 3$.  
Then for any transient subset $H$ of $G$ 
\[ p_{c,1}(G, H, \PP) \le p_c(G).\]  
\end{Prop}

Third we consider the case that $G$ is a tree $T$. 
\begin{Thm}
Let $T$ be an infinite tree and $H$ be a transient subset of $T$.  
Then,
$p_{c, 1}(T, H, \PP) = 1$.  
\end{Thm}

\begin{proof}
For $e \in E(T)$ and $x \in e$, 
we let $T_{e, x}$ be the connected subtree of $T$ such that $x \in V(T_{e,x})$ and $e \notin E(T_{e,x})$.  
Since $H$ is a transient subset, 
there are  an edge $e$ and a vertex $x \in e$ such that 
$T_{e, x}$ is a transient {\it subgraph} of $T$ and $V(H) \cap V(T_{e,x}) = \{x\}$.   
Then we can take an infinite path $(x_0, x_1, x_2, \dots)$ in $T_{e,x}$ 
such that 
$x_0 = x$, and for each $i \ge 0$, 
$\{x_i, x_{i+1}\} \in E(T_{\{x_{i-1}, x_{i}\},  x_i})$, and $T_{\{x_{i-1}, x_{i}\},  x_i}$ is a transient subgraph. 
If $p < 1$, then   
there is  a number $i$ such that $\U_p(H)$ does not intersect with $T_{\{x_{i-1}, x_{i}\},  x_i}$ $\Pp$-a.s.   
Hence $\U_p (H)$ is a transient subset of $T$, $\Pp$-a.s.     
\end{proof}

We do not give an assertion corresponding to Theorem \ref{epsilon-thm}. 
We are not sure that there is a recurrent subset such that the induced subgraph of it satisfies \eqref{near} in Lemma \ref{epsilon}. 

\subsection{The case that $H$ is the trace of the simple random walk}

\begin{proof}[Proof of Theorem \ref{recur-sub-main}]   
Let $o$ be the unit element of the group.        
Let \[ \theta(x) := \sum_{n \ge 0} P^{o}(S_n = x) \text{ and } \theta_p(x) := \sum_{n \ge 0} P^{o} \otimes \Pp \left(S_n \in C_x\right). \]   
We remark that $\theta(x) = \theta_0(x)$.  
First we show $p_c(G) \le p_{c,1}(G, H, \PP)$.  
Let $p < p_c(G)$.  
It follows from the Gaussian upper bound for heat kernel in \cite[Theorem 5.1]{HSC} and the finiteness of mean size of open cluster in the subcritical phase by \cite{M} that %explanation added !!!
\[ \theta_p(x) = O\left(d_{G}(o,x)^{2-d}\right).\]      
By following the proof of \cite[Theorem 6.5.10]{LL},   
\[ E^{P^o \otimes \Pp}\left[\sum_{x \in \U_p(H)} \theta(x)\right]  =  \sum_{x \in V(G)} \theta_p(x)\theta(x) = O\left(\sum_{x \in V(G)} d_{G}(o,x)^{4-2d}\right). \]

Using $d \ge 5$ and summation by parts,  
\[ \sum_{x \in V(G)} d_{G}(o,x)^{4-2d} < +\infty. \]

Hence \[ P^o \otimes \Pp\left(S_n \in \U_p(H), \text{ i.o. } n\right) = 0 \] and $p \le p_{c,1}(G, H, \PP)$.  
Thus we have 
\[ p_c(G) \le p_{c,1}(G, H, \PP).\]   
If $p_c(G) < 1$, then by Lemma \ref{Recur-cluster}, \[ p_c(G) \ge p_{c,2}(G, H, \PP).\]  
If $p_c(G) = 1$, this clearly holds. 
Thus we see assertion (i).   

We show assertion (ii) by following the proof of \cite[Theorem 6.5.10]{LL}.   
Let $S^{1}$ and $S^{2}$ be two independent simple random walks on $G$.    
Let \[ Z_{k} := \left| \left(B_G(o, 2^k) \setminus B_G(o, 2^{k-1})\right) \cap S^{1}\left([0, T^{1}_{V(G) \setminus B_G(o, 2^k)})\right) \cap S^{2}\left([0, T^{2}_{V(G) \setminus B_G(o, 2^k)})\right) \right|. \]
Let $E_{k}$ be the event that  $Z_k$ is strictly positive.   

In below, $c_i$, $1 \le i \le 8$, are positive constants depending only on $G$.  
It follows from a generalized Borel-Cantelli lemma that if \\
(1) \[ \sum_{k \ge 1} P^{o,o}(E_{3k}) = +\infty  \ \ \text{ and } \] 
(2) For some constant $c_1$ 
\[ P^{o,o}\left(E_{3k} \cap E_{3m}\right) \le c_1 P^{o,o}(E_{3k}) P^{o,o}(E_{3m}), \ \ k \ne m  \] 
hold, then $E_{3k}$ holds i.o. $k$, $P^{o,o}$-a.s. and assertion (ii) follows.

By \cite[Theorem 5.1]{HSC}, 
\[ c_2 \cdot d_{G}(o, x)^{2-d} \le \theta(x) \le c_3 \cdot d_{G}(o, x)^{2-d}.\]  
By Grigor\'yan and Telcs \cite[Proposition 10.1]{GT} the elliptic Harnack inequality holds.  
Therefore, we have (2).   

Now we show (1). 
\begin{align*} 
E^{P^{o,o}}[Z_k] &= \sum_{x \in B_G(o, 2^k) \setminus B_G(o, 2^{k-1})} P^o\left(T_x < T_{V(G) \setminus B_G(o, 2^k)}\right)^2 \\
&= c_4  \sum_{x \in B_G(o, 2^k) \setminus B_G(o, 2^{k-1})} \theta_{B_G(o, 2^k)}(x)^2.  
\end{align*}

Since $\theta_{B_G(o, 2^k)}(x) \ge c_5 2^{k(2-d)}$ for any $x \in B_G(o, 3 \cdot 2^{k-2}) \setminus B_G(o, 2^{k-1})$,   
\[ E^{P^{o,o}}[Z_k] \ge c_6 2^{k(4-2d)} \left|B_G(o, 3 \cdot 2^{k-2}) \setminus B_G(o, 2^{k-1}) \right|.  \]
Using this and an isoperimetric inequality (Cf. \cite[Theorem 7.4]{HSC}),       
\begin{equation}\label{one}
E^{P^{o,o}}[Z_k] \ge c_7 2^{k(4-d)}.
\end{equation}

We have 
\[ E^{P^{o,o}}[Z_k^2] = \sum_{x, y \in B_G(o, 2^k) \setminus B_G(o, 2^{k-1})} P^o\left(T_x \vee T_y  < T_{V(G) \setminus B_G(o, 2^k)}\right)^2.   \]

Since 
\begin{align*}
P^o\left(T_x \vee T_y  < T_{V(G) \setminus B_G(o, 2^k)}\right) &\le P^o\left(T_x \le T_y  < T_{V(G) \setminus B_G(o, 2^k)}\right) \\ 
&+ P^o\left(T_y \le T_x  < T_{V(G) \setminus B_G(o, 2^k)}\right)\\
&\le \theta(x)\theta(x^{-1} y) + \theta(y)\theta(y^{-1}x) \\
&= O\left(2^{k(2-d)} (1+d_G(x,y))^{2-d}\right), 
\end{align*} 
we have that by using summation by parts 
\[ \sum_{y \in B_G(o, 2^k) \setminus B_G(o, 2^{k-1})} (1+d_G(x,y))^{4-2d} = \begin{cases} O(2^{k}) \ \ \ d = 3, \\   O(k) \ \ \ d = 4.\end{cases} \] 
Therefore,    
\begin{equation}\label{two}
E^{P^{o,o}}[Z_k^2] = \begin{cases} O(4^{k}) \ \ \  d = 3, \\  O(k) \ \ \ d = 4.\end{cases}
\end{equation} 

Using (\ref{one}), (\ref{two}) and the second moment method, 
for $d = 3,4$,  
\[ P^{o,o}(E_k) = P^{o,o}(Z_k > 0) \ge \frac{E^{P^{o,o}}[Z_k]^2}{E^{P^{o,o}}[Z_k^2]} \ge \frac{c_8}{k}. \]
Thus we have (1).  
\end{proof}

\section{$\PP$ is being connected}

\begin{proof}[Proof of Theorem \ref{conn}]    
Let $\Pp^{H}$ be the product measure on $\{0,1\}^{E(G)}$ such that $\Pp^{H}(\omega_e = 1) = 1$ if $e \in E(H)$ 
and $\Pp^{H}(\omega_e = 1) = p$ if $e \notin E(H)$. 
Denote $x \leftrightarrow y$ if $x$ and $y$ are connected by an open path in this percolation model.    

Define $x \sim y$ if and only if  $\Pp^{H}(x \leftrightarrow y) = 1$.  
This is an equivalent definition.  
Let $[x]$ be the equivalent class containing $x$. 
Let $G^{\prime}$ be the quotient graph of $G$ by $\sim$. 
This is a connected graph which may have multi-lines.  
The number of edges between two vertices are finite.   

If $|V(G^{\prime})| = 1$, then \[ \Pp(\text{$\U_p(H)$ is connected}) = 1.\]   
Assume that $|V(G^{\prime})| \ge 2$, $V(G^{\prime}) = A \cup B$ and $A \cap B = \emptyset$.     
Let $A^{\prime}$ (resp. $B^{\prime}$) be a subset of $V(G)$ such that the equivalent class of each element is in $A$ (resp. $B$). 
Then \[ V(G) = A^{\prime} \cup B^{\prime}, \ A^{\prime} \cap B^{\prime} = \emptyset \text{ and } |A^{\prime}| = |B^{\prime}| = +\infty.\]   
By (TI),    
$E(A^{\prime}, B^{\prime}) = +\infty$.  
Since the number of edges between two vertices of $G^{\prime}$ are finite,  $E(A, B) = +\infty$. 

Define $p([x], [y])$ be the probability $[x]$ and $[y]$ are connected by an open edge with respect to the induced measure of $\Pp$ by the quotient map.     
Then 
\[ p\left([x], [y]\right) = 1 - (1-p)^{\left|E([x], [y])\right|}. \]
Hence    
\[ \sum_{[x] \in A, [y] \in B} p([x], [y]) \ge p \sum_{[x] \in A, [y] \in B} 1_{\left\{\text{$[x]$ and $[y]$ are connected by an edge of $G$}\right\}} = +\infty.\]   
By Kalikow and Weiss \cite[Theorem 1]{KW},      
\[ \Pp\left(\text{the random graph on $G^{\prime}$ is connected}\right) \in \{0,1\}.\]

Each connected component of $H$ is contained in an equivalent class, 
and conversely, 
each equivalent class contains each connected component of $H$, due to the percolating everywhere assumption. 
Therefore, $\U(H)$ is connected if and only if  the random graph on $G^{\prime}$ is connected.   
Hence, for any $p > 0$  
\[ \Pp(\text{$\U_p(H)$ is connected}) \in \{0,1\}\] 
and hence \[ p_{c,1}(G, H, \PP) = p_{c,2}(G, H, \PP).\]   
If the number of connected components of $H$ is finite, 
then  $E(A^{\prime}, B^{\prime}) <  +\infty$ for any decomposition $V(G^{\prime}) = A^{\prime} \cup B^{\prime}$. 
Therefore,   
$|V(G^{\prime})| = 1$, and hence,  
\[ p_{c,1}(G, H, \PP) = p_{c,2}(G, H, \PP) = 0.\]     
Thus we have assertion (i).     

Assume that $G$ does not satisfy (TI). 
Then there is  two infinite disjoint sets $A$ and $B$ such that $V(G) = A \cup B$ and $E(A, B) < +\infty$.    
Since $(A, E(A, A))$ and $(B, E(B, B))$ may have finite connected components,    
we modify $A$ and $B$.  
Let $\partial A$ and $\partial B$ be the inner boundaries of $A$ and $B$, respectively.  
For any vertex in $(A, E(A, A))$, there is  a vertex in $\partial A$ such that they are connected in $(A, E(A, A))$.         
Since $A$ and $B$ are infinite, there are  $a_0 \in \partial A$ and $b_0 \in \partial B$       
such that infinitely many vertices of $A$ are connected to $a_0$ in $A$ and infinitely many vertices of $B$ are connected to $b_0$ in $B$. 

Let 
\[  E_{A, B} := \left\{\{a,b\} \in E_{A} : b \leftrightarrow b_0 \text{ in } E(G) \setminus E_{A} \right\},  \]
where
\[ E_{A} := \left\{\{a,b\} \in E(A, B) : a \leftrightarrow a_0 \text{ in } A \right\}. \] 

Let $H$ be a subgraph of $G$ such that $E(H) = E(G) \setminus E_{A, B}$.    
Then $a_0$ and $b_0$ are not connected in $H$.  
Assume that there is  a vertex $x$ such that it is not connected to $a_0$ in $H$.      
Then consider a path $\gamma$ from $x$ to $b_0$ {\it in $G$}.   
Let $\{a, b\} \in E_{A, B}$ be the first edge which $\gamma$ intersects with $E_{A,B}$. 
Since $a \leftrightarrow a_0 \text{ in } A$, $\gamma$ pass $b$ before it pass $a$.   
There is  a path from $b$ to $b_0$ which does not pass any edges of $E_{A, B}$.   
Hence, there is  a path from $x$ to $b_0$ in $H$.  
Therefore, there are  just two connected components of $H$, and  
due to the choices of $a_0$ and $b_0$, 
they are both infinite. 
Thus $H$ is percolating everywhere.     

Since $E_{A, B}$ is finite, \[ p_{c,1}(G,H,\PP) = 0.\]   
Since $E_{A, B}$ is non-empty,  
\[ p_{c,2}(G,H,\PP) = 1.\]   
Thus we have assertion (ii).  
\end{proof}

%delated 

\section*{Acknowledgements} 
The author wishes to express his gratitude to N. Kubota for stimulating discussions, to I. Benjamini for pointing out Theorem \ref{Itai} and to H. Duminil-Copin for notifying me of the reference \cite{BeTa}.
The author also wishes to express his gratitude to a referee for comments.   
The author was supported by Grant-in-Aid for JSPS Research Fellow (24.8491) and Grant-in-Aid for Research Activity Start-up (15H06311) and for JSPS Research Fellows (16J04213).

\end{document}